\documentclass[10pt]{amsart}
\usepackage{amsmath,amssymb, amsthm} 

\theoremstyle{plain}
    \newtheorem{thm}{Theorem}

    \newtheorem{prop}{Proposition}[section]
    
    \newtheorem{lemma}[prop]{Lemma}
    
    \newtheorem{problem}[prop]{Problem}

\theoremstyle{definition}
    
\theoremstyle{remark}
    \newtheorem{rem}[prop]{Remark}

\def\tr{{\mathrm{tr}}}

\def\UH{\mathcal{UH}}

\def\be{\begin{equation}}
\def\ee{\end{equation}}

\def\bm{\begin{matrix}}
\def\em{\end{matrix}}

\def\sl{\mathrm{sl}}

\def\BB{{\mathfrak{B}}}
\def\bb{{\mathfrak{b}}}

\newcommand{\C}{\mathbb{C}}\newcommand{\R}{\mathbb{R}}\newcommand{\Q}{\mathbb{Q}}\newcommand{\Z}{\mathbb{Z}}\newcommand{\N}{\mathbb{N}}

\newcommand{\D}{\mathbb{D}}
\renewcommand{\P}{\mathbb{P}}
\newcommand{\cB}{\mathcal{B}}
\newcommand{\cE}{\mathcal{E}}\newcommand{\cF}{\mathcal{F}}\newcommand{\cH}{\mathcal{H}}

\newcommand{\cP}{\mathcal{P}}
\newcommand{\cR}{\mathcal{R}}
\newcommand{\cV}{\mathcal{V}}

\renewcommand{\setminus}{\smallsetminus}
\renewcommand{\emptyset}{\varnothing}
\newcommand{\SL}{\mathrm{SL}}

\newcommand{\id}{\mathit{id}}

\DeclareMathOperator*{\supp}{supp}

\def\H{\mathbb{{H}}}

\newcommand{\comm}[1]{}

\begin{document}

\title{Density of positive Lyapunov exponents for $\SL(2,\R)$ cocycles}

\author[A.~Avila]{Artur Avila}
\address{CNRS UMR 7586,
Institut de Math\'ematiques de Jussieu,
175 rue du Chevaleret,
75013, Paris, France}
\curraddr{IMPA,
Estrada Dona Castorina 110, 22460-320, Rio de Janeiro, Brasil} 
\urladdr{www.impa.br/$\sim$avila/}
\email{artur@math.sunysb.edu}

\begin{abstract}
We show that $\SL(2,\R)$ cocycles with a positive Lyapunov exponent are dense
in all regularity classes and for all non-periodic dynamical systems.
For Schr\"odinger cocycles, we show prevalence of potentials for which the
Lyapunov exponent is positive for a dense set of energies.
\end{abstract}

\date{\today}

\maketitle




\section{Introduction}

The understanding of the frequency of hyperbolic behavior is one of the
central themes in dynamics \cite {BV}.
Precise questions can be posed in several
levels, for instance, in the context of area-preserving maps:
\begin{enumerate}
\item Under suitable smoothness assumptions,
quasiperiodicity, and hence absence of any
kind of hyperbolicity, is non-negligible in a measure-theoretical sense
\cite {KAM} (under suitable smoothness assumptions),
\item In low regularity ($C^1$), failure of non-uniform hyperbolicity
(which here can be understood as positivity of the metric entropy)
is a fairly robust phenomenon in the topological sense \cite {B}.
\end{enumerate}
On the other hand, very little is understood about the emergence of
non-uniformly hyperbolic area-preseving diffeomorphisms
``in the wild'': indeed all known examples are either Anosov
or very carefully cooked modifications of uniformly hyperbolic systems.
While in low regularity one has so much flexibility to perturb that
one can neverthless show that positive metric entropy is dense in the
$C^1$-topology (e.g., by embedding renormalized ``cooked examples'' away
from Anosov maps\footnote {In fact any area-preserving map can be
approximated by one which is either Anosov or has an elliptic periodic
point.  In the latter case, it can be further perturbed so that some iterate
of the dynamics is the identity on some disk.  By further perturbation one can
create small periodic disks with high period and isometric behavior,
and a further perturbation allows one
to realize, in a small scale, any given smooth area-preserving map
of the unit disk (essentially by writing it as a composition of many
smooth area-preserving maps close to the identity), thus in
particular an example with positive metric entropy \cite {Ka}.}),
whether this is true
for reasonably smooth dynamics seems to be far beyond our current knowledge.

In this paper we consider the question of density of non-uniform
hyperbolicity in the considerable simpler context of $\SL(2,\R)$-cocycles.
Let $f:X \to X$ be a homeomorphism of a compact metric space, and let $\mu$
be a probability measure invariant by $f$.  If $A:X \to \SL(2,\R)$ is
continuous map,
we can define the {\it $\SL(2,\R)$-cocycle} $(f,A)$, acting on
$X \times \R^2$, by $(x,y) \mapsto (f(x),A(x) \cdot y)$.  The iterates of
the cocycle have the form $(f^n,A_n)$, and for $n \geq 1$ we have
$A_n(x)=A(f^{n-1}(x)) \cdots
A(x)$.  The Lyapunov exponent is defined by
\be \label {L(A)}
L(A)=L(f,\mu,A)=\lim_{n \to \infty} \frac {1} {n} \int \ln \|A_n(x)\|
d\mu(x).
\ee
The Lyapunov exponent is always non-negative, and a most fundamental
dynamical property of the cocycle $(f,A)$ is whether it is non-zero.
Examples of $\SL(2,\R)$-cocycles include the tangent dynamics $(f,Df)$ of
a $C^1$ conservative map of the torus $f:\R^2/\Z^2 \to \R^2/\Z^2$
(in this case the Lyapunov exponent coincides with the metric entropy), but
the cocycle setting we have the flexibility of ``perturbing $Df$
independently of $f$'', which considerably simplifies the analysis.

\begin{problem} \label {prob}

Can a cocycle with zero Lyapunov exponent be approximated by one with a
positive Lyapunov exponent?

\end{problem}

As usual in the cocycle setting, we take the point of view that
the base dynamics (given by $f$ and $\mu$) should be regarded as fixed
and only the fibered dynamics (given by $A$) should be allowed to
vary.  As mentioned before, the issue of regularity is of course usually
important in such approximation questions.
For instance, though we described a general setting above, when
$f:X \to X$ is, say, a smooth diffeomorphism of a smooth manifold, one often
wants to restrict considerations to smooth $A$, and correspondingly one
wishes to obtain approximations that are close as smooth maps.

Obviously high regularity density results are a priori harder to obtain than
low regularity ones.  Indeed in the lowest regularity, the continuous
category, Problem \ref {prob} can be easily resolved using {\it Kotani Theory}:
positive Lyapunov exponents are dense if and only if $f$ is not
periodic on $\supp \mu$ (that is, $f^k|\supp \mu \neq \id$ for every $k \geq
1$), see Lemma \ref {densC}.

Of course the continuous category allows for somewhat drastic
perturbations, and closely related problems have different answers in low
and high regularity.  For instance, if $f$ is a Bernoulli shift,
then there exist H\"older cocycles with positive
Lyapunov exponent which can be approximated by continuous cocycles with zero
Lyapunov exponent, but not by a H\"older one.\footnote {The H\"older
exponent $\alpha>0$ being fixed.  This happens for
instance if $f$ has two fixed points $p$ and $q$ and $0<\tr A(p)<2<\tr
A(q)<2+\epsilon(\alpha)$.  See \cite {BGV}.}
In fact the analysis of high regularity has generally
been restricted so far
to specific subclasses of dynamical systems (for instance, with
hyperbolic \cite {V} or quasiperiodic behavior \cite {FK}, \cite {A1}).

In this paper we will give a complete solution to Problem \ref {prob}: we will
both treat arbitrary dynamics (as defined above) and all usual regularity
classes (e.g., H\"older, Sobolev, smooth, G\'evrey, analytic).
To describe precisely what we have in mind as a regularity class,
it is convenient to introduce the
following concept.  A topological space
$\BB$ continuously included in $C(X,\SL(2,\R))$ is {\it ample} if
there exists some dense vector space
$\bb \subset C(X,\mathrm{sl}(2,\R))$, endowed with some finer (than uniform)
topological vector space structure, such that for every $A \in \BB$,
$e^b A \in B$ for every $b \in \bb$, and the map
$b \mapsto e^b A$ from $\bb$ to $\BB$ is continuous.  Notice that if $X$ is a
compact smooth or analytic manifold then the usual
spaces of smooth or analytic maps $X \to \SL(2,\R)$ are ample in our sense.

\begin{thm} \label {general}

Assume that $f$ is not periodic on $\supp \mu$, and let
$\BB \subset C(X,\SL(2,\R))$ be ample.  Then the Lyapunov exponent is positive
for a dense subset of $\BB$.

\end{thm}

\subsection{Schr\"odinger cocycles}

The most studied subclass of $\SL(2,\R)$ cocycles are the Schr\"odinger
cocycles, when
\be
A(x)=A^{(E-v)}(x)=\left (\bm E-v(x) & -1 \\ 1 & 0 \em \right )
\ee
for some $E \in \R$, $v \in C(X,\R)$.
Those appear naturally in the analysis of the Schr\"odinger operators of the
form
\be
(H u)_n=u_{n+1}+u_{n-1}+v(f^n(x)) u_n,
\ee
since an eigenfunction $H u=E u$ satisfies
$A^{(E-v)}_n(x) \left (\bm u_0 \\ u_{-1} \em \right )=
\left (\bm u_n\\u_{n-1} \em \right )$.
It will be convenient to use the short-handed notation
$L(v)=L(A^{(v)})$ for the Lyapunov exponent of Schr\"odinger cocycles.

The proof of Theorem \ref {general} can be easily adapted to show:

\begin{thm} \label {sch}

Assume that $f$ is not periodic on $\supp \mu$, and let
$V \subset C(X,\R)$ be a dense vector space
endowed with a finer topological vector space structure.  Then
for every $E \in \R$, the set of $v \in V$ such that
$L(E-v)>0$ is dense.

\end{thm}

However, since in the analysis of the
Schr\"odinger operator we need to vary the {\it energy}
$E$ for some fixed {\it potential} $v$,
it is natural to wonder whether the set of potentials
for which the Lyapunov exponent is positive for a dense subset of
energies is dense in $V$, and if so, whether it is ``large''.

Under the mild additional assumption that $V$ is a Baire space,
it is clear that Theorem \ref {sch} yields:
for a generic set (thus large topologically)
of $v \in V$, the Lyapunov exponent is positive for a dense subset
of energies.

However, it is often the case that a large set in the topological sense is
not large in the measure-theoretic sense, hence it makes sense to ask
whether this is what is going on here.  Of course we first need a notion of
``measure-theoretically large'': a common such notion is provided by the
concept of {\it prevalence} in a separable
Banach space $V$: a subset $P \subset V$ is
said to be prevalent if there exists a probability measure $\nu$ with
compact support in $V$ such that for every $v \in V$ and for almost every
$w$ with respect to $\nu$ we have $v+w \in P$.  This notion coincides with
``full Lebesgue measure'' in finite dimensional settings, see \cite {HSY}.

\begin{thm} \label {gauss}

Assume that $f$ is not periodic on $\supp \mu$, and let
$V \subset C(X,\R)$ be a dense vector space endowed with a finer separable
Banach space structure.  Then the set of $v \in V$ such that
$L(E-v)>0$ for a dense set of $E \in \R$ is prevalent.

\end{thm}

\begin{rem}

The dense
set of energies with a positive Lyapunov exponent provided by
Theorem \ref {gauss} is necessarily ``non-negligible everywhere''
in the sense that it intersects any open set in a positive measure set. 
Indeed, while the Lyapunov exponent is not always a continuous function of
$E$, it is measurably continuous everywhere in the sense that for every
$E_0 \in \R$
and for every $\epsilon>0$, $E_0$ is a Lebesgue density point of the set
$\{|L(E-v)-L(E_0-v)|<\epsilon\}$.  This can be obtained right
away from the Thouless formula representation $L(E-v)=\int \ln |E'-E|
dN(E')$, see e.g. \cite {AS},
where $N$ (known as the integrated density of states)
is the distribution of a compactly supported and continuous probability
measure on $\R$.

\end{rem}

\subsection{Further comments}

As mentioned above, previous progress on Problem \ref {prob}, in high
regularity, was restricted to specific classes of dynamical systems.  Two
important classes of examples are systems with periodic orbits of
arbitrarily large periods (allowing to deal with
situations arising in non-uniformly hyperbolic dynamics), and systems whose
behavior is sufficiently close to quasiperiodic.
In both cases, the most important techniques are provided by some form of
Kotani Theory \cite {K}, which imposes strict constrains on the
dynamics of cocycles with stably zero Lyapunov exponents.  Even when
succesful, the results were often highly depended on the regularity: for
translations of tori, density was proved in the analytic category, with the
inductive topology (which is not Baire), but not in Banach spaces
of analytic functions \cite {A1}, for the skew-shift, the smooth category
could be covered, but not the analytic one, etc
In this paper, we take a completely different approach that will allow us to
bypass the detailed understanding of the basis dynamics, relying
instead on a simple general property of the dependence of the Lyapunov
exponent with respect to parameters which we have discovered.

While the Lyapunov exponent is generally
a rather wild function of parameters (see \cite {B}), we will show that it
can be manipulated to yield regularized expressions with a nice dependence.
Perhaps the simplest example of a regularizing expression involving the
Lyapunov exponent is the \cite {AB} formula
\be \label {form}
\int_0^1 L(A R_\theta) d\theta=\int \ln \frac {\|A\|+\|A\|^{-1}} {2} d\mu.
\ee
Here $R_\theta$ is the rotation of angle $2 \pi \theta$.  In \cite {AK},
other expressions for regularizing the Lyapunov exponent have been shown to
exist, however all of them had an intrinsic non-local nature: to obtain
good dependence with respect to parameters, one combined
the Lyapunov exponents of ``topologically non-trivial'' families of
cocycles of a suitable type (e.g., perturbations of $\theta \mapsto A
R_\theta$).  This obviously makes such formulas
unsuitable to address approximation
problems, which are intrinsically local (while (\ref {form}) easily implies that
for most $A$ there exist some $\theta$ such that $L(A R_\theta)>0$,
it gives no information on their location).

In this paper we produce the first
local regularization expressions for the Lyapunov exponent with analytic
dependence on a parameter.  Let us describe
one such expression which is of use for the Schr\"odinger case.  Denote by
$\| \cdot \|$ the sup norm in either $C(X,\R)$ and for $r>0$
let $\cB(r)$ denote the open ball $\{\|w\|<r\}$ in $C(X,\R)$.

\begin{thm} \label {reg1}

For every $\epsilon>0$,
for every $v \in C(X,\R)$,
\be
\Phi_\epsilon(v,w)=\int_{-1}^1 \frac {1-t^2} {|t^2+2 i t+1|^2}
L(v+\epsilon (t+(1-t^2) w)) dt
\ee
defines an analytic function of $w$ in
$\cB(2^{-3/2})$.

\end{thm}


With such kind of result in hand, we can easily extend
the well known density of positive Lyapunov
exponents in the weakest regularity (continuous category), to
arbitrary regularity.


Let us note also that our methods can be applied to study
the prevalence of density (in energy) of positive Lyapunov exponents for
specific parametrized families of potentials.  The following example result
(related to the proof of Theorem \ref {gauss}) demonstrates this:

\begin{thm}

Let $J \subset C(X,\R)$ be a segment such that $L(E-v)>0$ is dense in
$(E,v) \in \R \times J$.  Then for almost every $v \in J$, there exists a
dense set of $E \in \R$ such that $L(E-v)>0$.

\end{thm}

This result is an immediate consequence of a more quantitative estimate,
see Theorem \ref {quantita}.

Let us conclude with some comments on the related issue of prevalence
of positive Lyapunov exponents: could it be possible
that positive Lyapunov exponents would be not merely dense, but
measure-theoretically typical?  This is well known not to be the case:
for smooth or analytic quasiperiodic cocycles with suitable Diophantine
conditions on the frequency, KAM theory (first advanced in this context by
Dinaburg-Sinai \cite {DS}) implies that zero Lyaunov exponents are
non-negligible (e.g., in the Kolmogorov sense: they do appear with positive
probability in open sets of parametrized families).
However, in other settings (for instance,
in the presence of some hyperbolicity and regularity, see \cite {V})
it has been established that zero Lyapunov exponents
are indeed negligible.

More generally,
Kotani Theory severely restricts the behavior of
typical cocycles with zero Lyapunov exponent (e.g., iterates not only grow
subexponentially, but stay bounded in the $L^1$-sense).  Other such
considerations eventually led to the formulation (see, e.g., \cite {J})
of the far-reaching
Kotani-Last Conjecture (for Schr\"odinger cocycles): except for almost
periodic potentials,
the Lyapunov exponent should be positive for almost every energy. 
Unfortunately this conjecture was recently shown to be false \cite {acnap},
and currently no reasonable replacement has been yet advanced.

\subsection{Outline of the paper}

We start most arguments with the analysis of the case of Schr\"odinger cocycles,
which is slightly more straightforward.
In section \ref {sec2} we give a proof of analyticity of a
regularized expression, Theorem \ref {reg2}, which is slightly more general
than Theorem \ref {reg1}, and then provide a non-Schr\"odinger version,
Theorem \ref {reg3}.  In section \ref {sec3} we go on to prove Theorems
\ref {sch} and \ref {general},
using analyticity to make the transition from local to
infinitesimal information and in the process allowing us to jump from the
easy continuous category to arbitrary regularity.  The slightly more technical
analysis of prevalence, Theorem \ref {gauss}, is carried out in section \ref {sec4}.

{\bf Acknowledgements:} This research was partially conducted during the
period A.A. was a Clay Research Fellow.

\section{Regularizing the parameter dependence of the Lyapunov exponent}
\label {sec2}

In this section we will prove
the following more general version of Theorem \ref {reg1}.

\begin{thm} \label {reg2}

For every $\epsilon>0$,
for every $v,v_0 \in C(X,\R)$, with $\eta=\inf v_0>0$,
\be
\Phi_\epsilon(v,v_0,w)=\int_{-1}^1 \frac {1-t^2} {|t^2+2 i t+1|^2}
L(v+\epsilon (t v_0+(1-t^2) w)) dt
\ee
defines an analytic function of $w$ in the ball
$\cB(\eta/2^{3/2})$.  Moreover, $w \mapsto \Phi_\epsilon(v,v_0,w)$
depends continuously (as an analytic function) on $v$.

\end{thm}


\subsection{$\SL(2,\C)$-cocycles and uniform hyperbolicity}

We will need a few basic facts about the dynamics of
$\SL(2,\C)$-cocycles (all results discussed below are well known, with
some proofs given for the convenience of the reader).
Similarly to $\SL(2,\R)$ cocycles, given
$A \in C(X,\SL(2,\C))$, one defines the
$\SL(2,\C)$ cocycle $(f,A)$ acting on $X \times \C^2$ by
$(f,A):(x,w) \mapsto (f(x),A(x) \cdot w)$.  We use the same notation for the
iterates $(f,A)^n=(f^n,A_n)$, and define the Lyapunov exponent $L(A)$ as
before, by (\ref {L(A)}).

Naturally the cocycle $(f,A)$ induces an action on $X \times \P\C^2$, which
we will not distinguish notationally.
Below we will use the spherical metric in $\P\C^2$.

We call $A \in C(X,\SL(2,\C))$
uniformly hyperbolic if there exists $C>0$, $\lambda>1$ such that
$\|A_n(x)\| \geq C \lambda^n$ for every $x$.  This is
equivalent (see, e.g., \cite {Y})
to the existence of a pair of
functions (the unstable and stable directions) $u=u_A$ and $s=s_A$ in
$C(X,\P\C^2)$ such that there exist $C_A>0$, $\lambda_A>1$
such that $A(x) \cdot u(x)=u(f(x))$, $A(x) \cdot s(x)=s(f(x))$,
$\|A_n(x) \cdot z\|,\|A_{-n}(x) \cdot w\| \leq C_A \lambda_A^{-n}$
for every $n \geq 1$ and all unit vectors $z \in s(x)$, $w \in u(x)$. 
Notice that the
unstable and stable directions are uniquely defined by those properties.
Notice that necessarily $u(x) \neq s(x)$ for every $x \in X$.

A {\it conefield} is an open set $U \subset X \times \P\C^2$
of the form $\bigcup_{x \in X} \{x\} \times U_x$
with $\emptyset \neq \overline U_x \neq \P\C^2$.
It is easy to see that if $m \in \P\C^2$ does not coincide with the stable
direction at $x$, then the distance between
$A_n(x) \cdot m$ and $u(f^n(x))$ decays exponentially fast, with the leading
constant depending only on a lower bound on the distance between $s(x)$ and
$m$.  In particular, if $U$ is a conefield and $(x,u(x)) \in U$,
$(x,s(x)) \notin \overline U$ for every $x \in X$,
then there exists $n \geq 1$ such that $(f^n(x),A_n(x) \cdot m) \in U$ for
every $(x,m) \in \overline U$.

Conversely, uniform hyperbolicity can be detected by a {\it conefield
criterion}: there exists a conefield $U$ and $n \geq 1$ such that
$(f^n(x),A_n(x) \cdot m) \in U$ for every $(x,m) \in \overline U$ if and
only if $A$ is uniformly hyperbolic (to get the direct implication, notice 
that the Schwarz Lemma implies that
$A_n(x):U_x \to U_{f^n(x)}$ strictly contracts the
Poincar\'e metric, uniformly in $x$, which readily gives
exponential growth of $\|A_k\|$).
Notice that in this case we necessarily have
$(x,u(x)) \in U$ and $(x,s(x)) \notin \overline U$ for every
$x \in X$.

Let $\UH \subset C(X,\SL(2,\C))$ denote the set of uniformly hyperbolic $A$. 
By the conefield criterion, it is clear that $\UH$ is an open set.

\begin{lemma}

For every $x \in X$, $A \mapsto
u_A(x)$ is a holomorphic function of $A \in \UH$.

\end{lemma}

\begin{proof}



Fix $A \in \UH$ and let $\epsilon_0$ be the infimum of the distance between
$u_A(x)$ and $s_A(x)$, $x \in X$.  Fix $0<\epsilon<\epsilon_0/2$,
and let $U$ be the conefield consisting of all $(x,m)$ such
that $m$ is $\epsilon$-close to $u(x)$.  Let $n \geq 1$ be such that
$(f^n(x),A_n(x) \cdot m) \in U$ for every $(x,m) \in \overline U$.  Let
$\cV \subset C(X,\SL(2,\C))$ be the set of all $A'$ such that
$(f^n(x),A'_n(x) \cdot m) \in U$ for every $(x,m) \in \overline U$.  Then
$\cV \subset \UH$, and it is an open neighborhood of $A$.

For $A' \in \cV$, $u_{A'}(x)$ is the limit of
$u^k_{A'}(x) \equiv A_{kn}'(f^{-kn}(x)) \cdot u_A(x)$.
Notice that for each $k \geq 1$, $A' \mapsto u^k_{A'}(x)$ is a holomorphic
function taking values in the hemisphere of $\P\C^2$ centered on $u_A(x)$. 
By Montel's Theorem, the limiting function $A' \mapsto u_{A'}(x)$ is
holomorphic.
\end{proof}

\begin{rem} \label {awa}

It follows from the proof of the previous lemma (take $\epsilon>0$
arbitrarily small) that if $A \in \UH$ and $A'$
is close to $A$ then $u_{A'}$ is uniformly close to $u_A$.  By the symmetric
argument, $s_{A'}$ is also uniformly close $s_A$.  In particular, the
distance between $u_{A'}(x)$ and $s_{A'}(x)$ is minorated for $A'$
near $A$.

\end{rem}



It is easy to see that the Lyapunov exponent is given by
\be
L(A)=\int_X \lambda(A(x),u_A(x)) d\mu=-\int_X \lambda(A(x),s(x)) d\mu,
\ee
where we define the expansion coefficient of the matrix $A \in \SL(2,\C)$ in
the direction $m \in \P\C^2$ by
$\lambda(A,m)=\ln \|A(x) \cdot z\|$, $z \in m$ a unit vector.

\def\HS{{\mathrm{HS}}}

\begin{lemma} \label {subha}

The Lyapunov exponent $A \mapsto L(A)$
is a plurisubharmonic function of $A \in C(X,\SL(2,\C))$. 
Moreover, restricted to the set of uniformly hyperbolic $A$, it is a
pluriharmonic function.

\end{lemma}

\begin{proof}

Plurisubharmonicity is immediate from the
alternative expression (a consequence of subadditivity) of the Lyapunov
exponent as the limit of the decreasing sequence of subharmonic functions
$\int_X \ln \|A_{2^n}\|_\HS d\mu$ (with $\| \cdot \|_\HS$ denoting the
Hilbert-Schmidt norm).

Let $A \in \UH$.  Let $h,v \in \P\C^2$ be horizontal and vertical directions,
and let $m \in \P\C^2 \setminus \{h,v\}$ be any other direction.
Let $B:X \to \SL(2,\C)$
be a bounded Borel function such that $B(x) \cdot u_A(x)=h$ and
$B(x) \cdot s_A(x)=v$.  Choose a small open neighborhood $\cV \subset
\UH$ of $A$, and let $\cB:\cV \times X \to \SL(2,\C)$ be such that
$\cB(A',x) \cdot u_{A'}(x)=h$, $\cB(A',x) \cdot s_{A'}(x)=v$ and $\cB(A',x)
\cdot B(x)^{-1} \cdot m=m$, and such that $\cB(A',x)$ is a continuous
funcion of $A'$ (the directions $u_{A'},s_{A'},B^{-1} \cdot m$ remain away
from each other by Remark \ref {awa}).  Then $\cB(A,x)$
is a bounded Borel function of $(A,x)$ and it is holomorphic in $A'$. 
Notice that $\cB(A',f(x)) \cB(A',x)^{-1}$ is diagonal, so it can be
written as $\begin{pmatrix} \gamma(A',x)&0\\0&\gamma(A',x)^{-1}
\end{pmatrix}$, where $\gamma:\cV \times X \to \C \setminus \{0\}$ is
a Borel function bounded away from $0$ and $\infty$ and $A' \mapsto
\gamma(A',x)$ is holomorphic for each $x \in X$.  We will show that
\be \label {gamma1}
L(A')=\int_X \ln |\gamma(A',x)| d\mu(x),
\ee
which implies that $A' \mapsto L(A')$ is pluriharmonic in $\cV$ (since for
each $x \in X$, $A' \mapsto \ln |\gamma(A',x)|$ is a bounded pluriharmonic
function).

Let $\gamma_n(A',x)=\prod_{k=0}^{n-1} \gamma(A',f^k(x))$.  Then
\be
\cB(A',f^n(x)) A'_n(x) \cB(A',x)^{-1}=\begin{pmatrix} \gamma_n(A',x)&0\\0&
\gamma_n(A',x)^{-1} \end{pmatrix}.
\ee
We have $|\gamma_n(A',x)| \geq C^{-1}_{A'} \lambda_{A'}^n \sup \|\cB\|^2$,
so for $n$ sufficiently large (independent of $x$), $|\gamma_n(A',x)|>1$. 
Thus
\be
\ln |\gamma_n(A',x)|-\ln \sup \|\cB\|^2 \leq
\ln \|A'_n(x)\| \leq \ln |\gamma_n(A',x)|+\ln \sup \|\cB\|^2,
\ee
and we get
\be
L(A')=\lim \frac {1} {n} \int_X \ln |\gamma_n(A',x)| d\mu(x),
\ee
and since $\int_X \ln |\gamma_n(A',x)| d\mu(x)=n \int_X \ln
|\gamma(A',x)| d\mu(x)$, (\ref {gamma1}) follows.
\end{proof}

The following gives a simple sufficient condition for uniform hyperbolicity
of Schr\"odinger cocycles.

\begin{lemma}

If $v \in C(X,\C)$ and $\Im v>0$
then $A^{(v)}$ is uniformly hyperbolic.

\end{lemma}

\begin{proof}

Let $\cH$ be the open hemisphere
of $\P\C^2$ centered on the line through
$\begin{pmatrix} i \\ 1 \end{pmatrix}$.  Then $A^{(v)}(x) \cdot m \in \cH$
for every $m \in \overline \cH$ which is not vertical, and $A^{(v)}(x)$ takes
the vertical direction to the horizontal direction.  Thus $A^{(v)}_2(x)
\cdot m \in \H$ for every $m \in \overline \cH$.  Thus the conefield
criterion is verified, with $U=X \times \cH$ and $n=2$.
\end{proof}

\subsection{Proof of Theorem \ref {reg2}}

We extend the notation $\| \cdot \|$ to denote the sup norm in $C(X,\C)$ (as
well as in $C(X,\R)$ as previously defined), and for $r>0$ we denote by
$\cB^\C(r)$ the open ball $\{\|w\|<r\}$ in $C(X,\C)$.

Fix $v,v_0 \in C(X,\R)$ with $\eta=\inf v_0>0$,
and for $z \in \C$ and $w \in C(X,\C)$, let
$\rho_w(z)=L(v+\epsilon z v_0+\epsilon^2 (1-z^2) w)$.

If $t=\epsilon e^{2 \pi i \theta}$ with $0<\theta<1/2$ and $w \in
\cB^\C(\eta/2 \epsilon)$ then
\begin{align}
\Im t v_0+(\epsilon^2-t^2) w &\geq \epsilon \eta
\sin 2 \pi \theta+\epsilon^2 \Im ((1-e^{4 \pi i \theta}) w)\\
\nonumber
&\geq \epsilon (\eta
\sin 2 \pi \theta-\epsilon |1-e^{4 \pi i \theta}| |w|)=
\epsilon \sin 2 \pi \theta (\eta-2 \epsilon |w|)>0.
\end{align}
It follows that for every $z \in \partial \D \cap \H$ (here $\H$ is the
upper half plane),
$w \mapsto \rho_w(z)$ is
pluriharmonic in $\cB^\C(\eta/2 \epsilon)$.

Notice that if $t=\epsilon (2^{1/2}-1)i$ and
$w \in \cB^\C(\eta/2 \epsilon)$ then
\be
\Im t v_0+(\epsilon^2-t^2) w \geq
\epsilon \eta (2^{1/2}-1)+2^{3/2} \epsilon^2 (2^{1/2}-1) \Im
w>0.
\ee
It follows that for $z=(2^{1/2}-1)i$, $w \mapsto \rho_w(z)$ is a
pluriharmonic function in the ball $\cB^\C(\eta/2^{3/2} \epsilon)$.

Let $\phi(z)=i (1-z)/(1+z)$, be the conformal map $\D \to \H$ which takes
$(1,i,-1)$ to $(0,1,\infty)$.  Then $\phi^{-1}(z)=-(z-i)/(z+i)$.
and $\psi(z)=\phi^{-1}(\phi(z)^{1/2})$
takes $\D$ to $\D \cap \H$.  Notice that $\psi(0)=(2^{1/2}-1) i$.

It follows that
\be \label {for1}
\rho_w(\psi(0))-\int_0^{1/2} \rho_w(\psi(e^{2 \pi i \theta})) d \theta
\ee
is a pluriharmonic function on $\cB^\C(\eta/2^{3/2}\epsilon)$.

Assume now that $t=\delta e^{2 \pi i \theta}$ with $0<\theta<1/2$,
$0<\delta<\epsilon$, and let us assume that $w \in \cB(\eta/2 \epsilon)$.
Since $w$ is now assumed to be real valued, we get
\be
\Im t v_0+(\epsilon^2-t^2) w \geq \delta
\sin 2 \pi \theta (\eta-2 \delta w \cos 2 \pi \theta)>0.
\ee
It follows that for every $w \in \cB(\eta/2 \epsilon)$,
$z \mapsto \rho_w(z)$ is harmonic
through $\D \cap \H$.  By the Poisson formula,
\be \label{for2}
\rho_w(\psi(0))=\int_0^1 \rho_w(\psi(e^{2 \pi i \theta})) d\theta
\ee
whenever $w \in \cB(\eta/2 \epsilon)$.

We conclude that
\be \label {for4}
\int_{1/2}^1 \rho_w(\psi(e^{2 \pi i \theta})) d\theta
\ee
is an analytic function on $w \in \cB(\eta/2^{3/2} \epsilon)$, since
(\ref {for1}) and (\ref {for4}) agree in this region (by (\ref {for2}))
and (\ref {for1}) has a
pluriharmonic extension.  A simple computation shows that
\be
\Phi_\epsilon(v,v_0,\epsilon w)=\int_{-1}^1 \frac {1-t^2} {|t^2+2 it+1|^2}
\rho_w(t) dt=\frac {\pi} {2}
\int_{1/2}^1 \rho_w(\psi(e^{2 \pi i \theta})) d\theta,
\ee
concluding the proof of analyticity.

For the continuous dependence with respect to $v$, notice that $(v,w)
\mapsto L(v+\epsilon z+\epsilon^2 (1-z^2))$ is jointly continuous on
$v \in C(X,\R)$ and $w \in \cB^\C(\eta/2^{3/2} \epsilon)$,
for every $z \in \partial \D \cap \H$ and also for
$z=(2^{1/2}-1) i$ (again by uniform hyperbolicity),
so that the pluriharmonic extension (\ref {for1}) of $w
\mapsto \frac {2} {\pi}
\Phi_\epsilon(v,v_0,\epsilon w)$ to $w \in \cB^\C(\eta/2 \epsilon)$,
depends continuously on $v$ and $w$.
\qed

\subsection{Non-Schr\"odinger case}

Let $\| \cdot \|_*$ denote the sup norm in $C(X,\sl(2,\R))$
and $C(X,\sl(2,\C))$, and for $r>0$ let $\cB_*(r)$ and $\cB^\C_*(r)$ be the
corresponding $r$-balls.

\begin{thm} \label {reg3}

There exists $\eta>0$ such that if $b \in C(X,\sl(2,\R))$ is
$\eta$-close to $\begin{pmatrix} 0&1\\-1&0 \end{pmatrix}$,
then for every $\epsilon>0$ and every $A \in C(X,\SL(2,\R))$,
\be
a \mapsto
\int_{-1}^1 \frac {1-t^2} {|t^2+2 i t+1|^2}
L(e^{\epsilon (t b+(1-t^2) a)} A) dt.
\ee
is a continuous function of $a \in \cB_*(\eta)$, which depends
continuously (as an analytic function) on $A$.

\end{thm}

\begin{proof}

The argument is mostly the same as for Theorem \ref {reg2}, so we just
explain what needs to be modified.  The key
point is to find $\eta>0$ such that we can check:
\begin{enumerate}
\item For $z \in \partial \D \cap \H$ and for $z=(2^{1/2}-1) i$,
$e^{\epsilon (z b+(1-z^2) a)} A$ is uniformly hyperbolic, provided
$a \in \cB^\C_*(\eta)$,
\item For $z \in \D \cap \H$, $e^{\epsilon (z b+(1-z^2) a)} A$ is
uniformly hyperbolic, provided $a \in \cB_*(\eta)$.
\end{enumerate}
Once this is done the remainder of the argument is unchanged.

Recall that $\cH$ is the open hemisphere of $\P\C^2$ centered at the line
through $\begin
{pmatrix} i\\1 \end{pmatrix}$.  Its boundary consist of the real directions,
$\partial \cH=\P\R^2$.  Fix a direction $m \in \P\R^2$, and let us consider
the path $m_\epsilon=e^{\epsilon (z b+(1-z^2) a)} \cdot m \in \P\C^2$.
We claim that
for $\eta>0$ small, and for $z$ and $a$ as in either cases
(1) or (2) above, the
derivative $\frac {d} {d\epsilon} m_\epsilon$ at $\epsilon=0$ points inside
$\H$.  This implies in particular that for every $\epsilon>0$ small,
$e^{\epsilon (z b+(1-z^2) a)}$ takes $\overline \cH$ into $\cH$.  By
iteration, $e^{\epsilon (z b+(1-z^2) a)}$ takes $\overline \cH$ into $\cH$
for every $\epsilon>0$, and since any $\SL(2,\R)$ matrix preserves $\cH$,
$e^{\epsilon (z b+(1-z^2) a)} A$ takes $\overline \cH$ into $\cH$ as well. 
Thus in either cases (1) or (2) above, uniform hyperbolicity is verified
through the conefield criterion, with $U=X \times \cH$.

It remains to check the above derivative estimate.  It will be convenient to
identity $\P\C^2$ with $\overline \C$, by taking the line through
$\begin{pmatrix} x\\y \end{pmatrix}$ to $\frac {x} {y}$.  Then $\cH$ is
identified with the upper half plane $\H$.

Writing $b=\begin {pmatrix} b_1&b_2\\
b_3&-b_1 \end{pmatrix}$, $a=\begin {pmatrix} a_1&a_2\\
a_3&-a_1 \end{pmatrix}$, we have
\be
\left .\frac {d} {d\epsilon} m_\epsilon \right |_{\epsilon=0}=
z (2 b_1 m+b_2-b_3 m^2)+(1-z^2)(2 a_1 m+a_2-a_3 m^2).
\ee
Notice that in either cases (1) or (2) above, $|\Im (1-z^2) a_j|=O(\eta \Im
z)$, $j=1,2,3$,
which gives
\be
\left .\Im \frac {d} {d\epsilon} m_\epsilon\right |_{\epsilon=0}=\Im z (1+m^2)(1+O(\eta)),
\ee
so choosing $\eta>0$ sufficiently small, $\frac {d} {d\epsilon} m_\epsilon$
has a positive imaginary component for $\epsilon=0$, so it points inside
$\cH$.

For the point ($m=\infty$ in the above identification)
missing from the above discussion, we can use a
different identification
taking the line through
$\begin{pmatrix} x\\y \end{pmatrix}$ to $\frac {-y} {x}$, which takes $\cH$
to $\H$ as before, but identifies the missing
point with $0$.  Then we have
\be
\left .\frac {d} {d\epsilon} m_\epsilon \right |_{\epsilon=0}=
-z b_3-(1-z^2) a_3,
\ee
which points inside $\cH$ as before.

This establishes the desired claim.
\end{proof}

\section{Density of positive Lyapunov exponents}
\label {sec3}

We assume throughout this section that $f|\supp \mu$ is not periodic.

\begin{lemma} \label {densC}

The set of $v \in C(X,\R)$ such that $L(v)>0$ is dense.

\end{lemma}

\begin{proof}

Let $P \subset X$ be the set of periodic orbits of $f$.

If $\mu(P)<1$, a stronger result is provided by \cite
{AD}: for generic $v \in C(X,\R)$, the set of $E \in \R$ such that
$L(E-v)>0$ has full Lebesgue measure.  (The results of \cite {AD} are
stated for the most important case of
$\mu$ ergodic, but the argument applies unchangend in this
larger generality.)

Assume that $\mu(P)=1$.  Below we will use well known
facts about periodic Schr\"odinger operators, see \cite {AMS} for a
discussion and further references, or \cite {lp} for a more recent
discussion close in style to this paper.
Let $P_k \subset X$ be the set of
periodic orbits of period $k \geq 1$.  Since $f|\supp \mu$ is not periodic,
$\cP_n=\bigcup_{k \leq n} P_k \neq \supp \mu$
for every $n \geq 1$.  Thus there is
arbitrarily large $n$ such that $\mu(\cP_n \setminus \cP_{n-1})>0$.
Choose such a large $n$, and let $x \in \supp \mu|\cP_n \setminus \cP_{n-1}$.
Then $x$ is a periodic orbit of period exactly $n$.
We can approximate
any $v \in C(X,\R)$ by some $v'$ which is constant in a compact neighborhood
$K$ of $\{f^k(x)\}_{k=0}^{n-1}$.  Moreover, we may assume that the set $\Sigma$
of $E \in \R$ such that $|\tr A^{(E-v')}_n(x)| \leq 2$ has exactly $n$ connected
components: This is because the values $v'(f^k(x))$, $0 \leq k \leq n-1$ may
be chosen independently and
``the spectrum of the generic periodic Schr\"odinger operator has all gaps
open'' (indeed, when changing one of the values of $v'(f^{n-1}(x))$ keeping the
others fixed, there are only at most
$2 n$ exceptional values for which $\Sigma$ has
less than $n$ connected components, see Claim 3.4 of \cite {lp} for an
argument).  Under this assumption,
the length of each connected component of $\Sigma$ is
bounded by $2 \pi/n$ (see Lemma 2.4 of \cite {lp}),
so there exists $E \in (-3 \pi/n,3 \pi/n)$ such that $|\tr
A^{(E-v')}(x)|>2$.  It follows that $L(E-v') \geq \mu(P_n \cap
K) \gamma/n$, with $\gamma=\lim_{k \to \infty}
\ln \|A^{E-v'}_{kn}(x)\|^{1/k}>0$.
\end{proof}

\begin{lemma} \label {bla3}

Let $v,v_0,w \in C(X,\R)$ and let $\epsilon>0$.  If
$L(v+w)>0$ then $\Phi_\epsilon(v,v_0,w)>0$.

\end{lemma}

\begin{proof}

By Lemma \ref {subha}, $\gamma(t)=L(v+\epsilon (t v_0+(1-t^2) w))$ is a
subharmonic function in $\C$.

By subharmonicity, if $\gamma(t)=0$ for almost every $t \in (-1,1)$, then
$\gamma(t)=0$ for every $t \in (-1,1)$.  So if $\gamma(0)=L(v+w)>0$ we
must have $\Phi_\epsilon(v,v_0,w)>0$ as well.
\end{proof}

\subsection{Proof of Theorem \ref {sch}}

We must show that for every $\delta>0$, $E \in \R$, $v_1 \in V$, there
exists $v_2 \in V$ with $\|v_2-v_1\|_V<\delta$ and $L(E-v_2)>0$.

Let $v=E-v_1$ and let $v_0 \in V$ be arbitrary with $\eta=\inf v_0>0$.
Let $\epsilon>0$ be such that $\epsilon \|v\|_V<\delta/2$.
By Lemmas \ref {densC} and \ref {bla3}, there exists $w \in
\cB(\eta/2^{3/2})$ such that $\Phi_\epsilon(v,v_0,w)>0$.
Since $V$ is dense
in $C(X,\R)$ and $w' \mapsto \Phi_\epsilon(v,v_0,w')$ is continuous in
$\cB(\eta/2^{3/2})$ (by Theorem \ref {reg2}), we may assume that $w \in V$.

By Theorem \ref {reg2}, $\gamma(s)=\Phi_\epsilon(v,v_0,s w)$ is an analytic
function of $s \in [-1,1]$.  Since $\gamma(1)>0$, $\gamma(s)>0$
for every $s>0$ sufficiently small.  Choose
$0<s<\min \{1,\delta/(2\epsilon \|w\|_V)\}$ with $\gamma(s)>0$.  Then
there exists $t \in (-1,1)$ such that
$L(v+\epsilon (t v_0+(1-t^2) s w))>0$.  Letting
$v_2=v_1-\epsilon (t v_0+(1-t^2) s w)$, we get
$L(E-v_2)>0$ and
$\|v_2-v_1\|_V \leq \epsilon (\|v_0\|_V+s \|w\|_V)<\delta$ as desired.
\qed

\subsection{Proof of Theorem \ref {general}}

The proof is readily adapted from the proof of Theorem \ref {sch}, with
the key Theorem \ref {reg2} being replaced by Theorem \ref {reg3}, once we
provide appropriate corresponding statements
to Lemmas \ref {densC} and \ref {bla3}.
The corresponding statement to Lemma \ref {bla3} is that
$L(e^{\epsilon a} A)>0$ implies that
\be
\int_{-1}^1 \frac {1-t^2} {|t^2+2 i t+1|^2}
L(e^{\epsilon (t b+(1-t^2) a)} A) dt>0,
\ee
and the proof is the same.
The corresponding statement to Lemma \ref {densC} is that positive
Lyapunov exponents are dense in $C(X,\SL(2,\C))$, and
the argument goes along the same lines: If the set of
non-periodic points has positive $\mu$-measure then for a generic
$A$ and almost every $\theta$ one has $L(R_\theta A)>0$ (see Remark 4.3 of
\cite {AD}), and otherwise one uses that for a generic cocycle $A$ over a
periodic orbit of period $n$, $R_\theta A$ is uniformly hyperbolic for a
$O(1/n)$ dense set of $\theta$.
\qed

\section{Prevalence of density of positive Lyapunov exponents}
\label {sec4}

By an affine embedding of the Hilbert cube in a topological vector space $\cE$
we shall understand a continuous affine map
$D:[0,1]^\N \to \cE$ (we do not assume injectivity).
For such an affine embedding, the probability measure on $\cE$ obtained by
push forward of the product measure on $[0,1]^\N$ is denoted by $\rho_D$.

An affine embedding is called non-degenerate if its image is not contained
in a proper closed affine subspace of $\cE$.
We call a Borel set $N \subset \cE$ {\it negligible} if $\rho_D(N)=0$.

It is clear that if $\cE$ is a separable Banach space then the complement of a
negligible set $N$ is prevalent in $\cE$, according to the definition given in
the introduction.

\begin{rem}

In fact, it is possible to show that a negligible set in
a separable Banach space is always a {\it Gauss-null set}, see \cite {BL},
section 6.2.  The notion of Gauss-null sets is strictly stronger than that
of {\it Haar-null sets} (i.e., the
complement of a prevalent set) unless $\cE$ is finite dimensional.

\end{rem}

It is immediate from the definition that if $\cF \subset \cE$ is a dense subspace
which is endowed with a finer topological vector space structure, then for any
$\cE$-negligible subset $N$, $\cF \cap N$ is $\cF$-negligible.  Thus the following
result contains Theorem \ref {gauss}.

\begin{thm} \label {gauss1}

The set $N$ of all $v \in C(X,\R)$ for which the set of energies $E$ with
$L(E-v)>0$ is not dense in $\R$ is negligible.

\end{thm}

The proof of this theorem will be based on an analysis of affine
one-parameter families of potentials.

\begin{thm} \label {quantita}

Let $v \in C(X,\R)$, $w \in \cB(2^{-3/2})$
and let $\epsilon>0$.  If $L(-v-\epsilon w)>0$ then for almost every
$t \in (0,\epsilon)$ there exists $E \in (-2 \epsilon,2 \epsilon)$
such that $L(E-v-t w)>0$.

\end{thm}

Before proving this result, let us see how it leads to Theorem \ref {gauss1}.

{\it Proof of Theorem \ref {gauss1}.}

Fix $\epsilon>0$, and let $N_\epsilon$ be the set of
all $v \in C(X,\R)$ such that $L(E-v)=0$ for every $E \in
(-3 \epsilon,3 \epsilon)$.  Naturally
$N=\bigcup_{q \in \Q} \bigcup_{n \in \N}
(N_{2^{-n}}+q)$, so it is enough to show that $N_\epsilon$ is negligible for
every $\epsilon>0$ (the translate of a negligible set is obviously
negligible).  Fix a non-degenerate embedding $D:[0,1]^\N \to C(X,\R)$.

For $x \in [0,1]^\N$, let us denote by $S_k(x)$ the set of all $y \in
[0,1]^\N$ with the same $k$ initial coordinates, and let $G_k$ be the tangent
space to the smallest affine space containing $D(S_k(x))$.

We claim that for every $x \in [0,1]^\N$, there exists $k(x) \in \N$ and a
non-trivial segment $[x,y(x)] \subset S_k(x)$
such that $\Phi_\epsilon(-D(x),-D(x+t(y-x))+D(x))>0$ for every $0<t \leq 1$.
Indeed, if the conclusion does not hold then
$w \mapsto \Phi_\epsilon(-D(x),w)=0$ for every $w \in \bigcup_k G_k$ with
$\|w\|<2^{-3/2}$ (by analyticity, Theorem \ref {reg1}).  But $\bigcup_k G_k$ is
dense by non-degeneracy,
so by continuity of $w \mapsto \Phi_\epsilon(-D(x),w)$,
$\Phi_\epsilon(-D(x),w)=0$ for $w \in \cB(2^{-3/2})$.  This contradicts Lemmas
\ref {densC} and \ref {bla3}.

Let $k(x)$ and $y(x)$ be as in the claim and let
$w'(x)=(-D(y(x))+D(x))/\epsilon$.
Up to replacing $y(x)$ by $x+t (y(x)-x)$ for small $t>0$,
we may assume that $\|w'(x)\|<2^{-3/2}$.  By continuity in $v$ of
$\Phi_\epsilon(v,w'(x))$ (Theorem \ref {reg2}), there exists an open neighborhood
$\cV(x) \subset [0,1]^\N$ of $x$ such that $\Phi_\epsilon(-D(z),w'(x))>0$ for
every $z \in \cV(x)$.  Thus for every $z \in \cV(x)$, there exists $E_0 \in
(-\epsilon,\epsilon)$ and $0<t_0<\epsilon$ such that $L(-D(z)+
E_0+t_0 w'(x))>0$.  By Theorem \ref {quantita} (with $v=-D(z)+E_0$ and
$w=w'(x) t_0/\epsilon$), for almost every $0<t<t_0$, there exists
$E \in (-3 \epsilon,3 \epsilon)$ such that $L(-D(z)+E+tw'(x))>0$.  By
Fubini's Theorem, for almost every $z \in \cV(x)$ (with respect to the
product measure), there
exists $E \in (-3 \epsilon,3 \epsilon)$ such that $L(-D(z)+E)>0$,
i.e., $D(z) \notin N_\epsilon$.

Since $[0,1]^\N$ is compact, we can cover it by finitely many $\cV(x)$, and
we conclude that for almost every $z \in [0,1]^\N$, $D(z) \notin
N_\epsilon$, as desired.
\qed

The proof of Theorem \ref {quantita} will involve some preparation.
It will be convenient to consider a ``convolved'' variation of
the functional $\Phi_\epsilon$.  For $0<\delta<1$, let
\be
\Phi_{\epsilon,\delta}(v,w)=\int_0^1 \int_{-\delta}^\delta
\Phi_\epsilon(v+\epsilon a,b w) da db.
\ee
Let us note the following general estimate, which does not depend on
analyticity.

\begin{lemma} \label {w1}

For every
$0<\delta_0<\delta_1<1$, there exist $C_1,C_2>0$ such that for every
$\kappa>0$, if $Z \subset (0,1)$ is a countable
set such that for every
$0<r_0<1$ there exists $r<r_0<(1+\kappa) r$
such that $r \in Z$, then for every $\epsilon>0$ and
$v,w \in C(X,\R)$ we have
\be
\Phi_{\epsilon,\delta_0}(v,w) \leq \sum_{r \in Z}
C_2 \kappa r
\Phi_{\epsilon,\delta_1}(v+r \epsilon w,C_1 \kappa r w).
\ee

\end{lemma}

\begin{proof}

Let $\Delta \subset \R^2$ be the set of all $(E,t)$ such that $0 \leq t \leq
1-E^2$.  Let $\nu$ be the absolutely continuous measure supported on
$\Delta$ with density
$\frac {d \nu} {dE dt}=\frac {1} {|1+2 i E+E^2|^2}$, $\gamma_\delta$ be the
linear (one-dimensional Hausdorff) measure on $[-\delta,\delta] \times
\{0\}$, and let $\nu_\delta=\nu * \gamma_\delta$.  Let $T_r(E,t)=(E,t+r)$,
$S_s(E,t)=(E,s t)$ and $\nu_{\delta,r,s}=(T_r)_* (S_s)_* \nu_\delta$.

We can write
\be
\Phi_{\epsilon,\delta}(v+r \epsilon w,s w)=\int L(v+\epsilon (E+t w))
d\nu_{\delta,r,s}(E,t).
\ee
In order to establish the desired
estimate, it is enough to show that
\be \label {Z}
\nu_{\delta_0} \leq \sum_{r \in Z} C_2 \kappa r (T_r)_* (S_{C_1 \kappa r})_*
\nu_{\delta_1}
\ee
for appropriate choices of $C_1$ and $C_2$.

The measure $\nu$ has density bounded above by $1$
and bounded below by $1/8$ through $\Delta$.  Thus the measure
$\nu_{\delta_0}$, whose support is contained in the rectangle
$[-1-\delta_0,1+\delta_0] \times [0,1]$, has density bounded by $2 \delta_0$.
On the other hand, there exists
$C_1=C_1(\delta_0,\delta_1) \geq 1$ such that
the density of $\mu_{\delta_1}$ is
at least $\frac {\delta_1-\delta_0} {10}$ in the rectangle
$[-1-\delta_0,1+\delta_0] \times [0,C_1^{-1}]$.  So
$C_1 \kappa r (T_r)_*
(S_{C_1 \kappa r})_* \nu_{\delta_1}$ has density at least $\frac
{\delta_1-\delta_0} {10}$ in the rectangle $\cR_r=[-1-\delta_0,1+\delta_0]
\times [r,(1+\kappa) r]$.
Letting $C_2=20 \frac {\delta_0} {\delta_1-\delta_0} C_1$, we see that
the right hand side of (\ref {Z}) has density at least $2 \delta_0$ in
$\bigcup_{r \in Z} \cR_r$, which by hypothesis contains
$[-1-\delta_0,1+\delta_0] \times (0,1)$.  The result follows.
\end{proof}

The next lemma uses in a simple way the analyticity of $\Phi_\epsilon$, as
well as continuity with respect to parameters.

\begin{lemma}

Let $\epsilon>0$, $v,w \in C(X,\R)$ be such that
\begin{enumerate}
\item $t \mapsto \Phi_\epsilon(v,t w)$ is non-identically vanishing near
$t=0$,
\item For every $\xi>0$, the set $X_\xi$
of $t \in [-\xi,\xi]$ such that
$L(v+t w+E)=0$ for every $E \in [-2 \epsilon,2 \epsilon]$ has positive
Lebesgue measure.
\end{enumerate}
Then for every $\xi>0$, there exists $C>0$, a compact set
$K \subset [-\xi,\xi]$ of positive Lebesgue measure, a non-trivial closed
interval $J \subset [-\epsilon,\epsilon]$,
a continuous function $c:J \times K \to \R_+$, and an even integer $k \geq
2$, such that for every $(E,t) \in J \times K$, we have
\be
|\Phi_\epsilon(v+E+t w,s w)-c(E,t) s^k| \leq C s^{k+1}, \quad |s| \leq 1.
\ee

\end{lemma}

\begin{proof}

The function
\be \label {s}
s \mapsto \Phi_\epsilon(v+E+t w,s w)
\ee
is analytic in a
neighborhood of $s=0$ for each $(E,t) \in \R^2$, and depends continuously on
$(E,t)$ (as an analytic function).  At $(E,t)=(0,0)$, this function is
non-identically vanishing.  Select $0<\xi<\epsilon$
so that non-vanishing holds for $(E,t) \in
[-\xi,\xi]^2$.  At each $(E,t) \in [-\epsilon,\epsilon] \times X_\xi$
we have $\Phi_\epsilon(v+E+t w,0)=0$.

By non-vanishing and analyticity of (\ref {s}),
for each $(E,t) \in [-\xi,\xi] \times X_\xi$ there exists
an even integer $k(E,t) \geq 2$ and $c(E,t)>0$ such that
$\Phi_\epsilon(v+E+t w,s w)=c(E,t) s^{k(E,t)}+O(s^{k(E,t)+1})$.  Using also
the continuity (as an analytic function)
of (\ref {s}) with respect to $(E,t)$, we see that
$k(E,t)$ is an upper-semicontinuous function.  Let $W_\xi \subset X_\xi$ be
the set of
Lebesgue density points, and for $l \geq 2$ even, let $M^l_\xi \subset
[-\xi,\xi] \times W_\xi$ be the set of $(E,t)$ with $k(E,t)=l$.  Let $k_0$ be
minimal such that $M^{k_0}_\xi$ is non-empty.  By upper-semicontinuity,
$M^{k_0}_\xi$ is open in $[-\xi,\xi] \times W_\xi$ so it has positive
Lebesgue measure.

Let $(E_0,t_0) \in M^{k_0}_\xi$, and let $\delta>0$ be such that if
$(E,t) \in [-\xi,\xi] \times W_\xi$ and $|t-t_0| \leq \delta$, $|E-E_0| \leq
\delta$, then $(E,t) \in M^{k_0}_\xi$.  Let $K \subset W_\xi \cap
[t_0-\delta,t_0+\delta]$ be a compact set of positive Lebesgue measure and
let $J=[-\xi,\xi] \cap [E_0-\delta,E_0+\delta]$. 
Then for $(E,t) \in J \times K$, $\Phi_\epsilon(v+E+t w,s w)=c(E,t)
s^k+O(s^{k+1})$.  By continuity (as an analytic function)
of (\ref {s}) with respect to $(E,t)$, the
function $c(E,t)$ is continuous over $J \times K$ and the error term is
uniform over $J \times K$.
\end{proof}

\begin{lemma} \label {shad}

Let $\epsilon>0$, $v,w \in C(X,\R)$ be such that $t \mapsto
\Phi_\epsilon(v,t w)$ is non-identically vanishing near $t=0$.  Then for
almost every
small $t$, there exists $E \in [-2 \epsilon,2 \epsilon]$ with $L(v+E+t
w)>0$.

\end{lemma}

\begin{proof}

If the result did not hold, then $\epsilon,v,w$ satisfy the hypothesis of
the previous lemma.  Let $J$,
$K$, $C$, $c$ and $k$ be as in the previous lemma, let $t_1$ be a Lebesgue
density point of $K$, and let $E_1$ be the center of $J$.  Let
$v_1=v+E_1+t_1 w$, and let $0<\delta_0<\delta_1<\min \{1,|J|/2\}$.
It follows that if $c_1>0$ is sufficiently small
($c_1<\frac {1} {k+1} \int_{-\delta_0}^{\delta_0}
c(E_1+a,t_1) da$) then for every $s>0$ sufficiently small,
\be
\Phi_{\epsilon,\delta_0}(v_1,s w)>c_1 s^k
\ee
On the other hand, if $c_2>0$ is sufficiently large, then for every $s>0$
sufficiently small, if $t \in K$ then
\be
\Phi_{\epsilon,\delta_1}(v+E_1+t w,s w) \leq c_2 s^k.
\ee
Fix such constants $c_1$ and $c_2$, and let $C_1$ and $C_2$ be the constants
of Lemma \ref {w1}.

Since $t_1$ is a Lebesgue density point of $K$, for every $\kappa>0$
we can select a sequence $s_i>0$, $i \geq 1$, such that
$t_1+\epsilon s_i \in K$ and $\kappa/2<\frac {s_i} {s_{i+1}}-1<\kappa$.

Let $s>0$ be small and let $i_0>0$ be minimal with
$s_i<s$.  By Lemma \ref {w1}, we get
\be
\Phi_{\epsilon,\delta_0}(v_1,s w) \leq \sum_{i \geq i_0}
C_2 s_i s^{-1} \kappa
\Phi_{\epsilon,\delta_1}(v_1+\epsilon s_i w,C_1 \kappa s_i w).
\ee
This gives
\be
c_1 s^k<\sum_{i \geq i_0} C_2 s_i s^{-1} \kappa
c_2 (C_1 \kappa s_i)^k,
\ee
that is
\be \label {c_1}
c_1 C_2^{-1} c_2^{-1} C_1^{-k}
\kappa^{-k-1}<\sum_{i \geq i_0} \left (\frac {s_i} {s}
\right )^{k+1}.
\ee
On the other hand, clearly
\be
\sum_{i \geq i_0} \left (\frac {s_i} {s} \right )^{k+1} \leq
\sum_{j \geq 0} \left (1-\frac {\kappa} {2} \right )^{j (k+1)}
\leq 2 \kappa^{-1},
\ee
which contradicts (\ref {c_1}) for small $\kappa>0$.
\end{proof}

{\it Proof of Theorem \ref {quantita}.}
By Theorem \ref {reg1}, for every $t \in (0,1)$,
$s \mapsto \Phi_\epsilon(-v-t \epsilon w,-s (1-t) w)>0$
is an analytic function of $s \in [-1,1]$, and
by Lemma \ref {bla3} and the hypothesis, it is non-zero at $s=1$.
By Lemma \ref {shad}, for every $t \in (0,1)$, for almost every small $s$,
there exists $E \in (-2 \epsilon,2 \epsilon)$
such that $L(E-v-\epsilon(t+s (1-t)) w)>0$.  The result follows.
\qed

\end{document}